\documentclass[a4paper,reqno]{amsart}

\usepackage{amssymb}

\usepackage{hyperref}

\usepackage{tikz-cd}

\usepackage{enumitem}
\setlist[enumerate]{leftmargin=*}

\newtheorem{theorem}{Theorem}[section]
\newtheorem{lemma}[theorem]{Lemma}
\newtheorem{corollary}[theorem]{Corollary}
\newtheorem{prop}[theorem]{Proposition}

\theoremstyle{definition}

\theoremstyle{remark}
\newtheorem{remark}[theorem]{Remark}

\numberwithin{equation}{section}

\newcommand{\ZZ}{\mathbb{Z}}
\newcommand{\NN}{\mathbb{N}}
\newcommand{\RR}{\mathbb{R}}
\newcommand{\CC}{\mathbb{C}}

\DeclareMathOperator{\sgn}{sgn}

\newcommand{\hnabla}{\nabla_{\!\varepsilon}}

\newcommand{\Rz}{\mathcal{R}}

\newcommand{\opL}{\mathcal{L}}

\newcommand{\trF}{\mathcal{F}}

\newcommand{\cS}{\mathcal{S}}
\newcommand{\cK}{\mathcal{K}}
\newcommand{\cP}{\mathcal{P}}

\newcommand{\id}{\mathrm{id}}

\newcommand{\Cv}{Cv}

\DeclareMathOperator{\imm}{Im}

\newcommand{\myth}{\omega_*}

\renewcommand{\colon}{\,:\,}

\newcommand{\pred}{\mathfrak{p}}
\renewcommand{\succ}{\mathfrak{s}}

\newcommand{\defeq}{\mathrel{:=}}

\DeclareMathOperator{\supp}{supp}

\begin{document}

\title{Riesz transform for a flow Laplacian on homogeneous trees}

\author[M. Levi]{Matteo Levi}
\address[M. Levi]{MaLGa Center \\ DIBRIS \\ Universit\`a di Genova \\ Via Dodecaneso 35 \\ 16146 Genova \\ Italy}
\email{m.l.matteolevi@gmail.com}

\author[A. Martini]{Alessio Martini}
\address[A. Martini]{Dipartimento di Scienze Matematiche ``Giuseppe Luigi Lagrange'' \\ Dipartimento di Eccellenza 2018-2022 \\ Politecnico di Torino \\ Corso Duca degli Abruzzi 24 \\ 10129 Torino \\ Italy}
\email{alessio.martini@polito.it}

\author[F. Santagati]{Federico Santagati}
\address[F. Santagati]{Dipartimento di Scienze Matematiche ``Giuseppe Luigi Lagrange'' \\ Dipartimento di Eccellenza 2018-2022 \\ Politecnico di Torino \\ Corso Duca degli Abruzzi 24 \\ 10129 Torino \\ Italy}
\email{federico.santagati@polito.it}

\author[A. Tabacco]{Anita Tabacco}
\address[A. Tabacco]{Dipartimento di Scienze Matematiche ``Giuseppe Luigi Lagrange'' \\ Dipartimento di Eccellenza 2018-2022 \\ Politecnico di Torino \\ Corso Duca degli Abruzzi 24 \\ 10129 Torino \\ Italy}
\email{anita.tabacco@polito.it}

\author[M. Vallarino]{Maria Vallarino}
\address[M. Vallarino]{Dipartimento di Scienze Matematiche ``Giuseppe Luigi Lagrange'' \\ Dipartimento di Eccellenza 2018-2022 \\ Politecnico di Torino \\ Corso Duca degli Abruzzi 24 \\ 10129 Torino \\ Italy}
\email{maria.vallarino@polito.it}

\thanks{Work partially supported by the MIUR project ``Dipartimenti di Eccellenza 2018-2022'' (CUP E11G18000350001) and the Project ``Harmonic analysis on continuous and discrete structures'' (bando Trapezio Compagnia di San Paolo CUP E13C21000270007). The first-named author also acknowledges the financial support of the AFOSR project FA9550-18-1-7009 (European Office of Aerospace Research and Development). The second-named author gratefully acknowledges the financial support of Compagnia di San Paolo through the ``Starting Grant'' programme. The authors are members of the Gruppo Nazionale per l'Analisi Matema\-tica, la Probabilit\`a e le loro Applicazioni (GNAMPA) of the Istituto Nazionale di Alta Matematica (INdAM)}


\begin{abstract}
We prove the $L^p$-boundedness, for $p \in (1,\infty)$, of the first order Riesz transform associated to the flow Laplacian on a homogeneous tree with the canonical flow measure. This result was previously proved to hold for $p \in (1,2]$ by Hebisch and Steger, but their approach does not extend to $p>2$ as we make clear by proving a negative endpoint result for $p = \infty$ for such operator.
We also consider a class of ``horizontal Riesz transforms'' corresponding to differentiation along horocycles, which inherit all the boundedness properties of the Riesz transform associated to the flow Laplacian, but for which we are also able to prove a weak type $(1,1)$ bound for the adjoint operators, in the spirit of the work by Gaudry and Sj\"ogren in the continuous setting.
The homogeneous tree with the canonical flow measure is a model case of a measure-metric space which is nondoubling, of exponential growth, does not satisfy the Cheeger isoperimetric inequality, and where the Laplacian does not have spectral gap.
\end{abstract}

\maketitle

\section{Introduction}

Let $T$ be a locally finite tree, which is a connected graph with no cycles where each vertex $x$ has a finite number $q(x)+1$ of neighbours. We identify $T$ with its set of vertices and equip it with the standard graph distance $d$, counting the number of edges along the shortest path connecting two vertices. We fix a reference point $o\in T$ and set $|x| \defeq d(x,o)$. A ray is a half-infinite geodesic, with respect to the distance $d$, emanating from $o$, and the natural boundary $\partial T$ of $T$ is identified with the family of rays. We choose a \emph{mythical ancestor} $\myth \in \partial T$ and consider the horocyclic foliation it induces on the tree: for each vertex $x$ there exists a unique integer index $\ell(x)$, which we call the \emph{level} of $x$, indicating to which horocycle the vertex belongs. The level function is given by $\ell(x) = d(o,x\wedge \myth)-d(x,x \wedge \myth)$, where $x \wedge \myth$ denotes the closest point to $x$ on the ray $\myth$. For each vertex $x$ we define its predecessor $\pred(x)$ as the only neighbour vertex such that $\ell(\pred(x)) = \ell(x)+1$, while $\succ(x)$ will denote the set of the remaining neighbours, the successors of $x$, whose level is $\ell(x)-1$. We introduce a partial order relation on $T$ by writing $x \geq y$ if $d(x,y) = \ell(x)-\ell(y)$.

A \emph{flow} on $T$ is a function $m$ satisfying the flow condition
\begin{equation}\label{eq: flow}
  m(x) = \sum_{y \in \succ(x)} m(y) \qquad \forall x \in T.
\end{equation}
Flows, which are common objects in Operations Research and Computer Science, turn out to have interesting properties also from a Harmonic Analysis point of view. Indeed, $p$-harmonic functions on trees can be characterized as appropriate nonlinear potentials of flows, see \cite{CL}. For a more wide-ranging account on the importance of flows in Probability and Analysis on trees, we refer the reader to \cite{LP}.

In this note we are interested in \emph{flow measures}, which are positive flows.
A flow $m$ is said to be \emph{canonical} if it distributes mass uniformly among the successors of each point, i.e., if $m(x) = q(x) \, m(y)$, for every $x \in T$, $y \in \succ(x)$. Up to normalization, the canonical flow is unique: we will refer to the one satisfying $m(o) = 1$ as \emph{the} canonical flow, and denote it with the letter $\mu$.

\smallskip

In the sequel we will deal with the homogeneous tree $T = T_q$, on which $q(x)=q$ for some integer $q\geq 2$ and every $x\in T_q$, equipped with the canonical flow measure $\mu(x) = q^{\ell(x)}$. A systematic analysis of ``singular integrals'' on $(T,\mu)$ was initiated in a remarkable paper by Hebisch and Steger \cite{hs}, where they developed an \emph{ad hoc} Calder\'on--Zygmund theory and studied the boundedness properties of spectral multipliers and the Riesz transform associated with a suitable Laplacian $\opL$, which we shall call the \emph{flow Laplacian}. Specifically, the flow Laplacian $\opL$ on $(T,\mu)$ can be written as
\[
\opL = \frac{1}{2} \, \nabla^* \, \nabla;
\]
here $\nabla$ denotes the \emph{flow gradient} on $(T,\mu)$, defined by
\begin{equation}\label{eq:vgrad}
\nabla f(x) = f(x) - f(\pred(x))
\end{equation}
for all $f : T \to \CC$ and $x \in T$, while $\nabla^*$ denotes the adjoint of $\nabla$ with respect to the $L^2(\mu)$-pairing. The Riesz transform $\Rz$ on $(T,\mu)$ can be then defined as
\[
\Rz = \nabla \opL^{-1/2}.
\]
In Section \ref{sec: lap e riesz} a more extensive discussion of the definition of the flow Laplacian is given, and the arising notion of Riesz transform is compared with other notions appearing in the literature. In \cite{ATV2, ATV1} an atomic Hardy space $H^1(\mu)$ and a space $BMO(\mu)$ adapted to $(T,\mu)$ were introduced and studied, and in \cite{San} the characterization of the Hardy space in terms of the Riesz transform $\Rz$ was proved to fail, i.e., the atomic Hardy space $H^1(\mu)$ is strictly contained in the space of integrable functions on $(T,\mu)$ whose Riesz transform is integrable.

In this note we aim at completing the study of the boundedness properties of the Riesz transform $\Rz$ on $(T,\mu)$. By \cite[Theorem 2.3]{hs} and \cite{ATV1}, $\Rz$ is of weak type $(1,1)$, bounded on $L^p(\mu)$ for $p\in (1,2]$, and bounded from $H^1(\mu)$ to $L^1(\mu)$: this follows from the fact that the integral kernel of $\Rz$ satisfies an ``integral H\"ormander condition'' adapted to this setting (see \eqref{eq: hormander} below). The problem of the $L^p$-boundedness for $p\in (2,\infty)$ was left open in \cite{hs} and we solve it in the following theorem.

\begin{theorem}\label{riesztr}
The Riesz transform $\Rz$ is bounded on $L^p(\mu)$ for $p \in (1,\infty)$.
\end{theorem}

We also show that $\Rz$ does not map $L^\infty(\mu)$ to $BMO(\mu)$ (see Proposition \ref{h1l1} below). This fact in particular shows that the integral kernel of the adjoint operator $\Rz^*$ does not satisfy the aforementioned integral H\"ormander condition, and therefore new ideas are required in order to prove the boundedness of the Riesz transform on $L^p(\mu)$, $p\in (2,\infty)$.

Our strategy is based on the observation that, since we know from \cite{hs} that $\Rz$ is $L^p$-bounded for $p \in (1,2]$, the $L^p$-boundedness of $\Rz$ for $p \in [2,\infty)$ is equivalent to the $L^p$-boundedness for $p \in (1,\infty)$ of the skew-symmetric part $\Rz-\Rz^*$. As it turns out, the operator $\Rz-\Rz^*$ has a simpler form, which is especially evident when the operator is lifted from the tree $T$ to the product $\Omega \times \ZZ$, where $\Omega=\partial T\setminus\{\myth\}$ is the ``punctured boundary'' of $T$. Indeed, remarkably, the lifted operator has the form $\id_\Omega \otimes (\Rz_\ZZ - \Rz_\ZZ^*)$, where $\Rz_\ZZ$ is the discrete Riesz transform on $\ZZ$, while $\id_\Omega$ is the identity operator on functions on $\Omega$. It is well-known \cite{ADP,HSC} that $\Rz_\ZZ$ is a Calder\'on--Zygmund operator on $\ZZ$, whence one easily deduces that $\id_\Omega \otimes (\Rz_\ZZ - \Rz_\ZZ^*)$ is $L^p$-bounded on $\Omega \times \ZZ$ for $p \in (1,\infty)$, and these strong type bounds transfer to $\Rz - \Rz^*$ too.

As a matter of fact, the same argument also gives the weak type $(1,1)$ boundedness of the lifted operator $\id_\Omega \otimes (\Rz_\ZZ - \Rz_\ZZ^*)$. However, this information per se does not appear to yield a corresponding weak type endpoint result for $p=1$ for the adjoint Riesz transform $\Rz^*$, whose validity remains an open problem.

\medskip

The study of the first-order Riesz transform $\Rz$ associated with the flow Laplacian $\opL$ on the homogeneous tree $T$ can be thought of as a discrete counterpart of the analysis of first-order Riesz transforms associated with a distinguished Laplacian $\opL_G$ on the so-called $ax+b$-groups $G$, developed in \cite{GS, hs, MaVa, Sj, SV2}.
In the latter context the natural gradient $\nabla_G$ is vector-valued, and the operator $\Rz_G = \nabla_G \, \opL_G^{-1/2}$ can be thought of as the \emph{vector of Riesz transforms}, whose components are the (first-order, scalar-valued) Riesz transforms on $G$; more specifically, corresponding to whether the component under consideration is in the direction of $a$ or $b$ in the $ax+b$-group, one speaks of a \emph{vertical} or a \emph{horizontal} Riesz transform on $G$.
We point out that the discrete Riesz transform $\Rz = \nabla \opL^{-1/2}$ on $T$ studied in this paper, despite being scalar-valued, should be thought of as an analogue of the vector of Riesz transforms $\Rz_G$ in the continuous setting, as the flow gradient $\nabla$ is comparable (at least, as far as weak or strong type bounds are concerned) with the ``modulus of the (full) gradient'' on $T$ (see Proposition \ref{p: gradient_equivalence} below).

In the aforementioned works on $ax+b$-groups, the $L^p$-boundedness for $p \in (1,2]$ of the full vector of Riesz transforms $\Rz_G$ was established, together with weak type $(1,1)$ and $H^1\to L^1$ endpoints. However, as far as we know, for $p>2$ the only currently available boundedness result concerns the horizontal Riesz transform on the smallest $ax+b$-group, for which Gaudry and Sj\"ogren in \cite{GS} proved the $L^p$-boundedness for all $p \in (2,\infty)$, as well as the weak type $(1,1)$ boundedness of the adjoint operator. In contrast, no analogous results for vertical Riesz transforms appear to be available, and, a fortiori, the $L^p$-boundedness for $p>2$ of the vector of Riesz transforms $\Rz_G$ appears to be so far an open problem.

This comparison with the continuous setting provides further reasons of interest for our Theorem \ref{riesztr}, as the $L^p$-boundedness result for $p>2$ that we obtain here appears to have no continuous counterpart in the literature on $ax+b$-groups. As it turns out, an approach similar in spirit to the one developed here can be applied to the study of $L^p$-boundedness properties for $p>2$ of Riesz transforms on $ax+b$-groups, eventually yielding that $\Rz_G$ is indeed $L^p$-bounded for all $p \in (1,\infty)$; details on this will appear elsewhere \cite{AMprep}.

\medskip

Motivated by the lack of an endpoint result at $p=1$ for the adjoint Riesz transform $\Rz^*$, and by the study of the horizontal Riesz transforms in the continuous setting, in this paper we also consider another class of Riesz transforms on $T$, which we shall also call \emph{horizontal Riesz transforms}. Specifically, for any given bounded function $\varepsilon : T \to \CC$ with the property that
\[
\sum_{y \in \succ(x)} \varepsilon(y) = 0 \qquad\forall x \in T,
\]
we define the associated \emph{horizontal gradient} $\hnabla$ as
\[
\hnabla f(x) = \frac{1}{q} \sum_{y \in \succ(x)} \varepsilon(y) f(y),
\]
and the corresponding horizontal Riesz transform $\Rz_\varepsilon = \hnabla \, \opL^{-1/2}$. At an intuitive level, one could think of the flow gradient $\nabla$ in \eqref{eq:vgrad} as (discrete) differentiation in the direction of the flow; instead, the horizontal gradient $\hnabla$ differentiates along horocycles, thus somewhat orthogonally to the flow.

This intuition is correct up to a point, as the flow gradient $\nabla$, as already mentioned, is comparable to a ``full'' gradient on $T$ for the purpose of weak or strong bounds; moreover, the step-$2$ differences implicit in the definition of $\hnabla$ (notice that distinct elements of $\succ(x)$ are at distance $2$ from each other) can be controlled by suitable combinations of the step-$1$ differences in $\nabla$. Correspondingly, any weak or strong $(p,p)$ bound for $\Rz$ transfers to $\Rz_\varepsilon$, and the same is true for the respective adjoint operators. While we are not able to determine whether $\Rz^*$ is of weak type $(1,1)$, nevertheless we manage to establish (see Theorem \ref{t: Respilonstar} below) that $\Rz_\varepsilon^*$ is.
Once again, it is not possible to prove this weak type endpoint result by directly using the aforementioned Calder\'on--Zygmund theory on $(T,\mu)$, as $\Rz_\varepsilon^*$ does not map $H^1(\mu)$ into $L^1(\mu)$ whenever $\varepsilon$ is nontrivial (see Proposition \ref{p:horiz_h1l1} below), and an \emph{ad hoc} approach is needed. Indeed, the weak type bound for $\Rz_\varepsilon^*$ can be thought of as a counterpart of the result by Gaudry and Sj\"ogren \cite{GS} for horizontal Riesz transforms in the continuous setting, and our proof is significantly inspired by theirs.

\medskip

It is important to point out that the metric space $(T,\mu)$ is an adverse setting to study the problem. Indeed, in \cite{LSTV}, the authors prove that flow measures fail to satisfy the Cheeger isoperimetric property, and do not satisfy the doubling condition, because they have exponential growth. It is well known that harmonic analysis in nondoubling settings presents major difficulties. In particular, extensions of the theory of singular integrals and of Hardy and BMO spaces have been considered on various metric measure spaces not satisfying the doubling condition, but fulfilling some measure growth assumptions or some geometric conditions, such as the isoperimetric property (see, e.g., \cite{CMM, MMV, NTV, T, To, Tol03, Ve}).

The boundedness of the Riesz transform on graphs has been the object of many investigations in recent years. In \cite{BaRu, fe, Ru2, Ru} the authors obtained various boundedness results for Riesz transforms on graphs satisfying the doubling condition and some additional conditions, expressed either in terms of properties of the measure or estimates for the heat kernel. In \cite{CeMe} Celotto and Meda showed that the Riesz transform associated with the combinatorial Laplacian is bounded from a suitable Hardy type space to $L^1$ on graphs with the Cheeger isoperimetric property. In the recent paper \cite{CCH} the authors obtained the $L^p$-boundedness of Riesz transform for the so-called bounded Laplacians on any weighted graph and any $p\in (1,\infty)$; however, the latter results are proved only under the assumption of positive spectral gap. We remark once again that $(T,\mu)$ is nondoubling and does not satisfy the Cheeger isoperimetric property. Moreover, the flow Laplacian $\opL$, which is a bounded Laplacian in the sense of \cite{CCH}, does not have spectral gap (see Section \ref{sec: lap e riesz}). Hence, none of the above-mentioned results may be applied in our case.

In \cite{LSTV} the Calder\'on--Zygmund theory of \cite{hs}, as well as the Hardy and BMO spaces of \cite{ATV2, ATV1}, were generalized to trees of bounded degree with arbitrary locally doubling flows. While the formulation of the problem for more general trees and flows does not require any additional effort, extending the results presented here to these more general situations seems far from being trivial, mainly because of the lack of explicit formulas for the heat kernel. A different approach is probably needed, and we will possibly tackle this problem in future work.

\medskip

The paper is organized as follows. In Section \ref{sec: lap e riesz} we introduce the flow Laplacian $\opL$, the Riesz transform $\Rz$, and we recall a few properties of the heat kernel of $\opL$, including its relation with the heat kernel on $\ZZ$.
In Section \ref{s: lifting} we discuss the transference result from $\Omega \times \ZZ$ to $T$.
 In Section \ref{s: verticalRiesz} we prove the $L^p$-boundedness of the Riesz transform for $p\in (1,\infty)$ and we show a negative endpoint result for $\Rz$ and $p=\infty$. In Section \ref{s: horizontal} we introduce and study the boundedness of horizontal Riesz transforms.

\medskip

Along the paper, if $A$ and $B$ are two sets, we write $A^B$ to denote the set of all functions from $B$ to $A$.
Moreover, if $A$ is a set, we write $\chi_A$ for the characteristic function of $A$, and we write $\id_A$ for the identity map on the set $\CC^A$ of complex-valued functions on $A$.

\section{The homogeneous tree and the flow Laplacian}\label{sec: lap e riesz}

In this section we collect all the notation and the preliminary results that will be used to study the boundedness of Riesz transforms on the homogeneous tree $T$ with the measure $\mu$.

\subsection{Combinatorial and flow Laplacians}
The combinatorial Laplacian on $T$, which we denote by $\Delta$, is the probabilistic Laplacian associated to the simple nearest neighbour random walk on $T$ and is defined by
\begin{equation*}
  \Delta f(x)=\frac{1}{q+1}\sum_{y\sim x}\bigl(f(x)-f(y)\bigr)  \qquad \forall f\in \CC^T, x\in T;
\end{equation*}
here $\sim$ denotes the neighbouring relation between vertices of $T$.
The operator $\Delta$ is bounded and self-adjoint on $L^2(\#)$, that is, the $L^2$ space on $T$ with respect to the counting measure $\#$.

We denote by $\opL$ the natural Laplacian on $(T,\mu)$, which we call the \emph{flow Laplacian} and is given by
\begin{equation}\label{flowlapdef}
  \opL f(x) = f(x) - \frac{1}{2\sqrt q}\sum_{y\sim x} \frac{\mu(y)^{1/2}}{\mu(x)^{1/2}}f(y) \qquad  \forall f \in \CC^T,x\in T.
\end{equation}
This is precisely the Laplacian on $T$ studied in \cite{hs}. It is easily seen that the flow Laplacian can be expressed in terms of the combinatorial Laplacian as follows:
\begin{equation}\label{eq: relation laplacians}
  \opL= \frac{1}{1-b} \, \mu^{-1/2} \, (\Delta-bI) \, \mu^{1/2},
\end{equation}
where $b= (\sqrt{q}-1)^2/(q+1)$, and $\mu^{1/2}$ and $\mu^{-1/2}$ are thought of as multiplication operators. Clearly $\mu^{1/2} : L^2(\mu) \to L^2(\#)$ is an isomorphism, and it is well known (see for instance \cite{CMS}) that $b$ is the bottom of the spectrum of $\Delta$ on $L^2(\#)$, from which it immediately follows that $\opL$ is self-adjoint on $L^2(\mu)$ and has no spectral gap. Indeed, the spectrum of $\opL$ is precisely $[0,2]$, see \cite[Remark 2.1]{hs}. Equation \eqref{eq: relation laplacians} plays a fundamental role in proving our results, because it allows us to exploit some known formulas for the heat kernel of the combinatorial Laplacian.

Notice that we can write
\begin{equation}\label{eq: LSigma}
\opL=\id_T -(\Sigma+\Sigma^*)/2,
\end{equation}
where $\Sigma : \CC^T \to \CC^T$ is defined by
\[
\Sigma f(x)=f(\pred(x))  \qquad \forall f \in \CC^T, x\in T,
\]
while $\Sigma^*$ is its adjoint with respect to the $L^2(\mu)$-pairing, given by
\[
\Sigma^*f(x) = \frac{1}{q}\sum_{y\in \succ(x)}f(y) \qquad \forall f \in \CC^T, x\in T.
\]
Such operators will often appear in the sequel and we shall summarize some of their properties in the following proposition.

\begin{prop}\label{SigmaSigma*}
The following hold:
\begin{enumerate}[label=(\roman*)]
\item\label{en:sigmasigmaproduct} for all $f,g \in \CC^T$,
\[
\Sigma^* (f \, \Sigma g) = g \, \Sigma^* f;
\]
\item\label{en:sigmasigmaproj} $\Sigma^*\Sigma= \id_T$;
\item\label{en:sigmasigmaiso} for every $p\in [1,\infty]$ the operator $\Sigma$ is an isometric embedding of $L^p(\mu)$ into itself, and also an isometric embedding of $L^{1,\infty}(\mu)$ into itself;
\item\label{en:sigmasigmabound} for every $p\in [1,\infty]$ the operator $\Sigma^*$ is bounded on $L^p(\mu)$ with norm $1$ and it is bounded on $L^{1,\infty}(\mu)$ with norm at most $q$.
\end{enumerate}
\end{prop}
\begin{proof}
For all $f,g \in \CC^T$ and $x \in T$,
\[
\Sigma^*(f \Sigma g)(x)=\frac{1}{q}\sum_{y\in \succ(x)} f(y) \Sigma g(y)=\frac{1}{q} \sum_{y\in \succ(x)} f(y) g(x) = g(x) \Sigma^* f(x),
\]
which proves part \ref{en:sigmasigmaproduct}. Taking $f \equiv 1$ in the previous identity yields part \ref{en:sigmasigmaproj}.

Consider now $p\in [1,\infty)$ and $f \in \CC^T$. Then
\[
\|\Sigma f\|^p_{L^p(\mu)} = \sum_{x\in T}|f(\pred(x))|^p \, q^{\ell(x)} = \sum_{x\in T}|f(\pred(x))|^p \, q^{\ell(\pred(x))-1} = \| f\|^p_{L^p(\mu)},
\]
proving the $L^p(\mu)$ result of part \ref{en:sigmasigmaiso} in the case where $p \in [1,\infty)$; the remaining case $p=\infty$ is analogous and easier. Furthermore, as $\mu$ is a flow measure, for all $\lambda > 0$,
\[
\mu \{ |\Sigma f| > \lambda \} = \mu (\pred^{-1}(\{ |f| > \lambda \})) = \mu \{ |f| > \lambda \},
\]
which proves that $\| \Sigma f\|_{L^{1,\infty}(\mu)} = \| f\|_{L^{1,\infty}(\mu)}$.

The $L^p$-boundedness statement in part \ref{en:sigmasigmabound} follows by part \ref{en:sigmasigmaiso} and duality. It remains to show the $L^{1,\infty}$-boundedness.
Given $\lambda >0$ and $f$ in $\CC^T$ we have that
\[\begin{split}
  \{x\in T \colon |\Sigma^*f(x)|>\lambda\}
		&= \biggl\{x\in T \colon \biggl|\sum_{y \in \succ(x)}f(y)\biggr|>q\lambda\biggr\} \\
		&\subseteq \{x\in T \colon \max_{y \in \succ(x)}|f(y)|>\lambda\} \\
		&=\bigcup_{j=1}^q \{x \colon |f(\succ_j(x))|>\lambda\},
\end{split}\]
where $\succ_j(x)$, $j=1,\dots,q$, is an enumeration of $\succ(x)$. It follows that
\[
  \mu( \{|\Sigma^*f|>\lambda\})\le \sum_{j=1}^q \mu \{|f \circ \succ_j|>\lambda\}
		\le {q} \mu\{|f|>\lambda\}\le q \frac{\|f\|_{L^{1,\infty}(\mu)}}{\lambda},
\]
because
\[
\mu \{|f \circ \succ_j|>\lambda\} = \sum_{x \in T} q^{\ell(x)} \chi_{\{|f(\succ_j(x))|>\lambda\}} = q\sum_{x \in T}q^{\ell(\succ_j(x))} \chi_{\{|f(\succ_j(x))|>\lambda\}}.
\]
Hence $\Sigma^*$ is bounded on $L^{1,\infty}(\mu)$ with norm at most $q$.
\end{proof}

\subsection{Gradient and Riesz transform}

The definition of Riesz transform depends on a notion of gradient on graphs, which is not unambiguous in the literature. Many authors, including Hebisch and Steger in \cite{hs}, define the ``modulus of the gradient'' of a function $f \in \CC^T$ as the vertex function
\begin{equation*}
  D f(x)=\sum_{y \sim x}|f(x)-f(y)| \qquad \forall x \in T,
\end{equation*}
and consequently the ``modulus of the Riesz transform'' as the sublinear operator $D \opL^{-1/2}$; as usual, fractional powers of the Laplacian are defined by means of the Spectral Theorem.

Here we find it natural and convenient to define the \emph{flow gradient} on $T$ as
\begin{equation*}
 \nabla f(x)= (\id_T -\Sigma) f(x) = f(x) - f(\pred(x)) \qquad \forall f \in \CC^T, x\in T.
\end{equation*}
Note that, by \eqref{eq: LSigma} and Proposition \ref{SigmaSigma*} \ref{en:sigmasigmaproj},
\[
\nabla^* \, \nabla = (\id_T-\Sigma^*) (\id_T-\Sigma) = 2 \opL,
\]
thus the flow gradient $\nabla$ is naturally associated with the flow Laplacian $\opL$, in that it allows one to write the latter in ``divergence form''.
We then define the Riesz transform on $(T,\mu)$ as the linear operator
\begin{equation*}
  \Rz f(x) = \nabla \opL^{-1/2} f(x) = \opL^{-1/2}f(x) - \opL^{-1/2} f(\pred(x)) \qquad \forall f \in \CC^T, x \in T.
\end{equation*}

We now show that the relevant boundedness properties of $\Rz$ are equivalent to those of the operator $D\opL^{-1/2}$ studied in \cite{hs}. On the basis of the following statement, we will be allowed to use the boundedness results from \cite{hs} for the modulus of the Riesz transform on $(T,\mu)$ as applying to $\Rz$ too.

\begin{prop}\label{p: gradient_equivalence}
For every $p \in [1,\infty]$,
\begin{equation*}
       \|\nabla f\|_{L^p(\mu)} \leq \| D f \|_{L^p(\mu)} \leq (1+q) \|\nabla f\|_{L^p(\mu)},
\end{equation*}
and
\begin{equation*}
      \| \nabla f \|_{L^{1,\infty}(\mu)} \leq \| D f \|_{L^{1,\infty}(\mu)} \leq (1+q)^2 \| \nabla f \|_{L^{1,\infty}(\mu)}.
\end{equation*}
\end{prop}
\begin{proof}
To prove the above statement for $L^p$ norms, recall that $\mu(x)=q\mu(y)$ if $y\in \succ(x)$; so, when $p < \infty$,
\begin{equation*}
    \begin{split}
        \Vert \nabla f\Vert_{L^p(\mu)}^p
				&\le \| D f\|_{L^p(\mu)}^p \le (1+q)^{p-1} \sum_{x \in T} \sum_{y \sim x} |f(x)-f(y)|^p \,\mu(x) \\
    &=(1+q)^{p-1} \sum_{x \in T} \biggl( |f(x)-f(\pred(x))|^p \,\mu(x) + q \sum_{y \in \succ(x)} |f(x)-f(y)|^p \,\mu(y) \biggr)\\
    &=(1+q)^p  \Vert \nabla f\Vert_{L^p(\mu)}^p.
    \end{split}
\end{equation*}
The case $p=\infty$ is analogous and easier.

Finally, on the one hand it is clear that $\Vert \nabla f\Vert_{L^{1,\infty}(\mu)}\leq \Vert D f\Vert_{L^{1,\infty}(\mu)}$. On the other hand, for any $\lambda >0$,
\begin{multline*}
    \lbrace x \colon |D f(x)|>\lambda\rbrace \\
		\subseteq \biggl\lbrace x \colon |\nabla f(x)| > \frac{\lambda}{q+1} \biggr\rbrace \cup \biggl\lbrace x \colon \exists y \in \succ(x) \colon |f(x)-f(y)|> \frac{\lambda}{q+1} \biggr\rbrace,
\end{multline*}
from which it follows that
\begin{equation*}
    \lambda\mu(\lbrace x \colon |D f(x)| > \lambda \rbrace) \leq (q+1)^2 \Vert \nabla f\Vert_{L^{1,\infty}(\mu)},
\end{equation*}
as required.
\end{proof}

\subsection{Laplacian and Riesz transform on \texorpdfstring{$\ZZ$}{Z}}\label{sec: prelZ}

Let $\Delta_{\ZZ}$ denote the discrete Laplacian on $\ZZ$, namely,
\begin{align*}
    \Delta_{\ZZ} F(n) = F(n) - \frac{F(n+1)+F(n-1)}{2} \qquad \forall n \in \ZZ,
\end{align*} for every $F$ in $\CC^{\ZZ}$. Observe that $\Delta_{\ZZ} = \id_\ZZ - (\tau_1+\tau_{-1})/2$, where $\tau_k F(n) = F(n-k)$ is the translation by $k \in \ZZ$. We also introduce the discrete (step-1) gradient $\nabla_{\ZZ} = \id_\ZZ - \tau_{-1}$ and the associated Riesz transform on $\ZZ$ (also known as the ``discrete Hilbert transform''), formally defined as $\Rz_\ZZ = \nabla_{\ZZ} \, \Delta_{\ZZ}^{-1/2}$. We point out that
\[
\nabla_{\ZZ}^* = \id_\ZZ - \tau_1 = -\tau_1 \nabla_\ZZ
\]
and
\begin{equation}\label{eq:laplacianZZ}
\Delta_\ZZ = \frac{1}{2} \nabla_\ZZ^* \nabla_\ZZ = \frac{1}{2} \left( \nabla_\ZZ + \nabla_\ZZ^* \right).
\end{equation}

Many of the above identities are analogous to the ones obtained above for the flow Laplacian and the vertical gradient on $T$; this is natural, as $\ZZ$ can be thought of as the homogeneous tree $T_q$ with $q=1$. A crucial difference between the case $q=1$ considered here and the case $q\geq 2$ discussed above is that the translation operator $\tau_1$ on $\ZZ$ is invertible, with inverse $\tau_{-1}$, and in particular $\tau_1$ and $\tau_{-1}$ commute; the same does not hold for the operators $\Sigma$ and $\Sigma^*$ on $T = T_q$ for $q \geq 2$. More generally, all the operators that we introduced on $\ZZ$ ($\Delta_\ZZ$, $\nabla_\ZZ$, $\Rz_\ZZ$, and their adjoints) are translation-invariant and (due to the commutativity of $\ZZ$) commute pairwise.

Let us now consider the skew-symmetric part $\widetilde \Rz_\ZZ$ of the Riesz transform $\Rz_\ZZ$, namely,
\[
\widetilde \Rz_\ZZ = \Rz_\ZZ - \Rz_\ZZ^* = \widetilde\nabla_\ZZ \, \Delta_\ZZ^{-1/2},
\]
where
\[
\widetilde\nabla_\ZZ = \tau_1 - \tau_{-1};
\]
in other words, $\widetilde \Rz_\ZZ$ can be also thought of as a first-order Riesz transform on $\ZZ$, associated to the skew-symmetric step-$2$ gradient $\widetilde \nabla_\ZZ$.

We record here some useful properties of $\Rz_\ZZ$ and $\widetilde \Rz_\ZZ$.

\begin{prop}\label{p: rieszZZ}
Let $k_\ZZ$ and $\tilde k_\ZZ$ be the convolution kernels of $\Rz_\ZZ$ and $\widetilde \Rz_\ZZ$. Then
\begin{equation}\label{eq:rieszZZformulas}
k_\ZZ(n) = \frac{\sqrt{2}}{\pi} \frac{1}{n+1/2}, \qquad \tilde k_\ZZ(n) = \frac{2\sqrt{2}}{\pi} \frac{n}{n^2-1/4}
\end{equation}
for all $n \in \ZZ$. In particular, $k_\ZZ$ and $\tilde k_\ZZ$ are Calder\'on--Zygmund kernels, i.e., they satisfy the estimates
\begin{align}
\label{CZk}
  |k^{\ZZ}(n)| &\leq C (1+|n|)^{-1}, & |\nabla_{\ZZ}k^{\ZZ}(n)| &\leq C (1+|n|)^{-2}, \\
\label{CZktilde}
  |\tilde{k}^{\ZZ}(n)| &\leq C (1+|n|)^{-1}, & |\nabla_{\ZZ}\tilde{k}^{\ZZ}(n)| &\leq C (1+|n|)^{-2}
\end{align}
for some constant $C \in (0,\infty)$ and all $n \in \ZZ$. Moreover, the operators $\Rz_\ZZ$ and $\widetilde \Rz_\ZZ$ and their adjoints are of weak type $(1,1)$ and bounded on $\ell^p(\ZZ)$ for all $p \in (1,\infty)$.
\end{prop}
\begin{proof}
The explicit formula for $k_\ZZ$ can be found in \cite[Proposition 1]{ADP}; note that our $\Rz_\ZZ$ coincides with $-\sqrt{2} \, \mathcal{H}^+$ in the notation of \cite{ADP}. From the explicit formula, the estimates \eqref{CZk} are immediately verified; alternatively, as $\ZZ$ is commutative and finitely generated (thus of polynomial growth), one can invoke the more general theory of \cite[pp.\ 695--696]{HSC}, which also discusses weak type $(1,1)$ and $L^p$-boundedness properties. As $\widetilde \Rz_\ZZ = \Rz_\ZZ - \Rz_\ZZ^*$ and
\[
\tilde k_\ZZ(n) = k_\ZZ(n) - k_\ZZ(-n),
\]
the analogous results for $\widetilde \Rz_\ZZ$ follow.
\end{proof}

\subsection{Heat kernels on \texorpdfstring{$\ZZ$}{Z} and \texorpdfstring{$T$}{T}}\label{s: heat}

By translation-invariance, the heat semigroup $e^{-t\Delta_\ZZ}$ on $\ZZ$ is a convolution operator; we shall denote by $h^\ZZ_t$ ($t>0$) the corresponding convolution kernel on $\ZZ$.

We now move to the homogeneous tree $T$. Let $e^{- t \Delta}$ and $e^{- t \opL}$ be the heat semigroups of the combinatorial Laplacian $\Delta$ and of the flow Laplacian $\opL$ on $T$, respectively. We shall denote by $h_t$ and $H_t$ the associated heat kernels on the respective measure spaces on which the generators are self-adjoint and bounded, i.e.,
\begin{equation}\label{eq:heatkernelsT}
  e^{-t\Delta} f(x) = \sum_{y \in T} h_t(x,y) f(y), \quad e^{-t\opL} f(x) = \sum_{y \in T} H_t(x,y) f(y) \mu(y).
\end{equation}
By the Spectral Theorem and \eqref{eq: relation laplacians}, we obtain the following relation between the combinatorial and the flow semigroups,
\begin{equation}\label{eq: htHt}
  e^{-t\opL} = \mu^{-1/2} e^{b t/(1-b)} e^{-t\Delta/{(1-b)}} \mu^{1/2}.
\end{equation}
By means of this relation, we can deduce the following formula for $H_t$ from an analogous known formula for $h_t$.

\begin{prop}\label{p: heatkernelT}
For all $t>0$ and $x,y \in T$,
\begin{equation}\label{eq: H=qJ}
H_t(x,y) = q^{-\ell(x)/2} \, J_t(d(x,y)) \, q^{-\ell(y)/2},
\end{equation}
where, for all $n \in \NN$,
\[
J_t(n) = \sum_{k=0}^{\infty} q^{-(n+2k)/2} \, \widetilde{\nabla}_{\ZZ} h^{\ZZ}_t(n+2k+1).
\]
\end{prop}
\begin{proof}
From \cite[Proposition 2.5]{CMS} we know that
\begin{equation}\label{eq: heat_combinatorial}
h_t(x,y) = e^{-bt} q^{-d(x,y)/2} \sum_{k=0}^\infty q^{-k} \, \widetilde\nabla_\ZZ h^{\ZZ}_{t(1-b)}(d(x,y)+2k+1)),
\end{equation}
and the desired formula for $H_t$ easily follows from \eqref{eq:heatkernelsT} and \eqref{eq: htHt}.
\end{proof}

An important feature of the heat kernel formula \eqref{eq: H=qJ} is the fact that, apart from the factor $q^{-(\ell(x)+\ell(y))/2}$, the expression for $H_t(x,y)$ only depends on the distance $d(x,y)$ between the vertices $x$ and $y$; this ``almost-radiality'' of the heat kernel $H_t$ of the flow Laplacian is a counterpart of the radiality of the heat kernel $h_t$ of the combinatorial Laplacian given in \eqref{eq: heat_combinatorial}, which in turn is a consequence of the homogeneity of $T$.

Another crucial feature of the formula in Proposition \ref{p: heatkernelT} is the fact that it relates the heat kernels $H_t$ on $T$ and $h_t^\ZZ$ on $\ZZ$. As a consequence, by subordination, we can deduce an expression for the integral kernel of $\opL^{-1/2}$, relating it to the convolution kernel $\tilde k^\ZZ$ of the skew-symmetric part $\widetilde \Rz_\ZZ$ of the Riesz transform on $\ZZ$.

\begin{corollary}\label{c: negsquareroot}
The integral kernel of $\opL^{-1/2}$ has the form
\begin{equation}\label{eq:nsqroot_formula}
K_{\opL^{-1/2}}(x,y) = q^{-\ell(x)/2} \, G(d(x,y)) \, q^{-\ell(y)/2},
\end{equation}
where
\begin{equation}\label{eq: G}
G(n)=\sum_{k=0}^\infty q^{-(n+2k)/2} \, \tilde{k}^{\ZZ}(n+2k+1).
\end{equation}
In particular,
\begin{equation}\label{eq: Gdiff}
q^{n/2} [ G(n) - G(n+2) ] = \tilde k^\ZZ(n+1)
\end{equation}
for all $n \in \NN$.
\end{corollary}
\begin{proof}
From the subordination formula
\[
\opL^{-1/2} = \frac{1}{\sqrt{\pi}} \int_0^\infty e^{-t\opL} \frac{dt}{t^{1/2}}
\]
and \eqref{eq: H=qJ}, we deduce that the integral kernel of $\opL^{-1/2}$ has the form \eqref{eq:nsqroot_formula}, with $G$ given by
\[
G(n) = \frac{1}{\sqrt{\pi}} \int_0^\infty J_t(n) \frac{dt}{t^{1/2}} = \sum_{k=0}^{\infty} q^{-(n+2k)/2} \frac{1}{\sqrt{\pi}} \int_0^\infty \widetilde{\nabla}_{\ZZ} h^{\ZZ}_t(n+2k+1) \frac{dt}{t^{1/2}}.
\]
The analogous subordination formula applied to $\Delta_\ZZ$ in place of $\opL$ also gives
\[
\widetilde \Rz_\ZZ = \frac{1}{\sqrt{\pi}} \int_0^\infty \widetilde\nabla_\ZZ e^{-t\Delta_\ZZ} \frac{dt}{t^{1/2}},
\]
that is,
\begin{equation}\label{eq:oddrieszZZ}
\tilde k^\ZZ = \frac{1}{\sqrt{\pi}} \int_0^\infty \widetilde\nabla_\ZZ h^\ZZ_t \frac{dt}{t^{1/2}},
\end{equation}
and the desired expression \eqref{eq: G} for $G$ follows. The identity \eqref{eq: Gdiff} is an immediate consequence of \eqref{eq: G}.
\end{proof}

As we shall see, the identity in Corollary \ref{c: negsquareroot} will be crucial for us to deduce boundedness properties of $\Rz$ from those of $\Rz_\ZZ$. This deduction is made possible by the transference result discussed in the next section.

\section{Punctured boundary, disintegration and transference}\label{s: lifting}

Let $\Omega=\partial T \setminus \{\myth\}$, where $\myth$ is the mythical ancestor determining the direction of the flow.
For all $\omega \in \Omega$ and $n \in \ZZ$, we define $\omega_n \in T$ as the only vertex in the geodesic $[\omega,\myth]$ joining $\omega$ to $\omega_*$ such that $\ell(\omega_n) = n$. The mapping $\Omega \times \ZZ \ni (\omega,n) \mapsto \omega_n \in T$ is clearly surjective, and allows us to consider $T$ as a ``quotient'' of the product $\Omega \times \ZZ$.

Crucially, one can disintegrate the flow measure $\mu$ along this mapping, and consider it as the push-forward of a product measure on $\Omega \times \ZZ$. Namely, we can equip $\Omega$ with the measure $\nu$ such that,
if we set $\Omega_x = \{w \in \Omega \colon x \in [\omega,\myth]\}$, then
\[
\nu(\Omega_x)=\mu(x)=q^{\ell(x)}
\]
for all $x \in T$ (see \cite[Section 2]{CMS2}) and \cite[Formula (3.5)]{Veca}). An application of Fubini's Theorem then readily shows that
\begin{equation}\label{disintegration}
  \sum_{x \in T} f(x) \,\mu(x) = \int_\Omega \sum_{n \in \ZZ} f(\omega_n) \, d\nu(\omega)
\end{equation}
for all nonnegative or $\mu$-summable $f \in \CC^T$ (cf.\ \cite[Formula (3.1)]{CMS2}).

We now define the \emph{lifting operator} $\Phi : \CC^{T} \to \CC^{\Omega \times \ZZ}$ by
\begin{equation}\label{sigmaPhi}
		\Phi f(\omega,n) = f(\omega_n) \qquad \forall \omega\in\Omega, n\in\ZZ.
\end{equation}

\begin{prop}\label{p: PhiPhi*}
The following hold:
\begin{enumerate}[label=(\roman*)]
\item\label{en:phiphidisint} for every nonnegative or $\mu$-summable $f \in \CC^T$,
\[
  \sum_{x \in T} f(x) \, \mu(x) = \int_{\Omega\times \ZZ} \Phi f \, d(\nu \times \#);
\]
\item\label{en:phiphiiso} $\Phi$ is an isometric embedding from $L^p(\mu)$ to $L^p(\nu \times \#)$ for every $p\in [1,\infty]$, and also from $L^{1,\infty}(\mu)$ to $L^{1,\infty}(\nu \times \#)$;
\item\label{en:phiphistar} the adjoint map $\Phi^*$ is given by
\[
\Phi^*g(x) = \frac{1}{\nu(\Omega_x)} \int_{\Omega_x} g(\omega,\ell(x)) \, d\nu(\omega);
\]
\item\label{en:phiphibound} $\Phi^*$ maps $L^{p}(\nu \times \#)$ to $L^p(\mu)$ with norm equal to 1 for every $p \in [1,\infty]$, and moreover $\Phi^* \Phi = \id_T$;
\item\label{en:phiphiunbound} the map $\Phi \Phi^*$ is not bounded on $L^{1,\infty}(\nu\times \#)$.
\end{enumerate}
\end{prop}
\begin{proof}
Property \ref{en:phiphidisint} is just a rephrasing of \eqref{disintegration}. Property \ref{en:phiphiiso} follows from \ref{en:phiphidisint} and the fact that $|\Phi f|^p =\Phi(|f|^p)$ for all $p \in [1,\infty)$.

We now prove property \ref{en:phiphistar}. For every $f$ in $\CC^T$ and $g$ in $\CC^{\Omega\times \ZZ}$,
\[\begin{split}
  \int_{\Omega \times \ZZ} (\Phi f) \, \overline{g} \,d(\nu \times \#)
		&=\int_{\Omega} \sum_{n \in \ZZ} f(\omega_n) \, \overline{g}(\omega,n) \,d\nu(\omega)\\
		&=\sum_{n \in \ZZ} \sum_{x \colon \ell(x)=n} f(x) \int_{\Omega_x} \overline{g}(\omega, n) \,d\nu(\omega) \\
    &=\sum_{x \in T} f(x) \, \overline{\frac{1}{\nu(\Omega_x)} \int_{\Omega_x} g(\omega,\ell(x)) \,d\nu(\omega)} \, \mu(x),
\end{split}\]
whence we deduce the formula for $\Phi^*$.

As for part \ref{en:phiphibound}, the $L^p$-boundedness with norm $1$ of $\Phi^*$ follows by duality from part \ref{en:phiphiiso}, and the case $p=2$ of the latter also implies that $\Phi^* \Phi = \id_T$.

To prove \ref{en:phiphiunbound}, let us fix an element $\overline\omega \in \Omega_o$ (thus $\overline{\omega}_0 = o$), and note that $\{\Omega_{\overline{\omega}_n}\}_{n\in \ZZ}$ is a strictly increasing sequence of subsets of $\Omega$ with
\[
\bigcap_{n \in \ZZ} \Omega_{\overline{\omega}_n} = \{\overline{\omega}\}, \qquad \bigcup_{n \in \ZZ} \Omega_{\overline{\omega}_n} = \Omega.
\]
Let $F_o : \Omega \to \RR$ and $F : \Omega \times \ZZ \to \CC$ be defined by
\[
F_o = \sum_{n\leq 0} q^{-n} \chi_{\Omega_{\overline{\omega}_n}\setminus \Omega_{\overline{\omega}_{n-1}}} , \qquad F = F_o \otimes \chi_{\{0\}}.
\]
It is easy to see that
\[
\int_\Omega F_o \,d\nu = \sum_{n\leq 0} q^{-n}\nu( \Omega_{\overline{\omega}_n} \setminus \Omega_{\overline{\omega}_{n-1}}) = \sum_{n \leq 0} q^{-n} (q^n - q^{n-1}) = +\infty.
\]
Moreover, for every $\lambda >0$,
\[
\{(\omega,n) \colon |F(\omega,n)| > \lambda \} = \{\omega \colon F_o(\omega) > \lambda \} \times \{0\} = (\Omega_{\overline{\omega}_{n(\lambda)}} \setminus \{\overline{\omega}\}) \times \{0\} ,
\]
where $n(\lambda) = \max \{ n \leq 0 \colon n < \log_{q}(1/\lambda)\}$, so
\[
(\nu\times \#)(\{(\omega,n) \colon |F(\omega,n)| > \lambda \}) = q^{n(\lambda)} \leq \frac{1}{\lambda},
\]
and therefore $F \in L^{1,\infty}(\nu\times \#)$. Now, for every $\omega \in \Omega_o$,
\[
\Phi \Phi^* F(\omega,0) = \Phi^* F(o) = \frac{1}{\nu(\Omega_o)} \int_{\Omega_o} F_o \,d\nu = +\infty\,,
\]
which implies that $\Phi \Phi^* F$ does not belong to $L^{1,\infty}(\nu \times \#)$. This proves \ref{en:phiphiunbound}.
\end{proof}

An immediate consequence of the boundedness properties of the lifting operator is the following relation between weak and strong type bounds of operators on $T$ and on $\Omega \times \ZZ$.

\begin{prop}\label{p: liftingbounds}
Assume that $A$ and $\alpha$ are linear operators on $\CC^T$ and $\CC^{\Omega \times \ZZ}$ respectively.
\begin{enumerate}[label=(\roman*)]
\item\label{en:liftweak} $A$ is of weak type $(1,1)$ on $(T,\mu)$ if and only if $\Phi A \Phi^*$ is of weak type $(1,1)$ on $(\Omega \times \ZZ, \nu \times \#)$, and their norms are the same.
\item\label{en:liftstrong} For any $p \in [1,\infty]$, $A$ is $L^p(\mu)$-bounded if and only if $\Phi A \Phi^*$ is $L^p(\nu \times \#)$-bounded, and their norms are the same.
\item\label{en:liftstrongrev} For any $p \in [1,\infty]$, if $\alpha$ is $L^p(\nu \times \#)$-bounded, then $\Phi^* \alpha \Phi$ is $L^p(\mu)$-bounded, with norm not greater than that of $\alpha$.
\end{enumerate}
\end{prop}
\begin{proof}
Part \ref{en:liftstrongrev}, as well as the ``only if'' implications in parts \ref{en:liftweak} and \ref{en:liftstrong}, follow immediately by the boundedness properties of $\Phi$ and $\Phi^*$ discussed in Proposition \ref{p: PhiPhi*} \ref{en:phiphiiso}-\ref{en:phiphibound}. As for the reverse implication in part \ref{en:liftweak}, it is enough to observe that
\[
\| A f \|_{L^{1,\infty}(\mu)} = \| \Phi A \Phi^* \Phi f \|_{L^{1,\infty}(\nu \times \#)} \leq \| \Phi A \Phi^* \|_{L^1 \to L^{1,\infty}} \| f \|_{L^1(\mu)}
\]
as $\Phi$ is an isometric embedding and $\Phi^* \Phi = \id_T$ by Proposition \ref{p: PhiPhi*}; a completely analogous argument proves the remaining implication in part \ref{en:liftstrong}.
\end{proof}

\begin{remark}\label{rem:lifting}
The implication in part \ref{en:liftstrongrev} of Proposition \ref{p: liftingbounds} cannot in general be reversed. Indeed, according to part \ref{en:liftstrong}, $\Phi^* \alpha \Phi$ is $L^p(\mu)$-bounded if and only if $\Phi \Phi^* \alpha \Phi \Phi^*$ is $L^p(\nu \times \#)$-bounded; clearly the latter would follow from the $L^p$-boundedness of $\alpha$, but is not equivalent to it, as the ``averaging operator'' $\Phi \Phi^*$ may reduce $L^p$ norms. Similarly, by part \ref{en:liftweak}, $\Phi^* \alpha \Phi$ is of weak type $(1,1)$ if and only if $\Phi \Phi^* \alpha \Phi \Phi^*$ is; however $\Phi \Phi^*$ is unbounded on $L^{1,\infty}$ (see Proposition \ref{p: PhiPhi*} \ref{en:phiphiunbound}), so the weak type $(1,1)$ of $\alpha$ in general does not imply the analogous property for $\Phi^* \alpha \Phi$.
\end{remark}

We now define the \emph{shift operator} $\sigma: \CC^{\Omega \times \ZZ} \to \CC^{\Omega \times \ZZ}$ by
\begin{equation}\label{sigma}
\sigma g(\omega,n) = g(\omega,n+1) \qquad \forall g\in \CC^{\Omega \times \ZZ},\omega\in\Omega, n\in\ZZ.
\end{equation}
Moreover, for every $n\in\ZZ$ we define
\begin{equation}\label{tildeSigma}
\tilde\Sigma_n = \begin{cases}
\Sigma^n &\text{if } n \geq 0,\\
(\Sigma^*)^{-n} &\text{if } n<0.
\end{cases}
\end{equation}
The maps $\Phi$, $\sigma$ and $\Sigma$ form a commutative diagram,
\[
\begin{tikzcd}
\CC^{\Omega \times \ZZ} \arrow{r}{\sigma} & \CC^{\Omega \times \ZZ}\\
\CC^{T} \arrow{r}{\Sigma} \arrow{u}{\Phi} & \CC^T \arrow{u}{\Phi},
\end{tikzcd}
\]
as discussed in the following proposition.

\begin{prop}
The following hold:
\begin{enumerate}[label=(\roman*)]
\item\label{en:sigma_lift} $\sigma \Phi = \Phi \Sigma$;
\item\label{en:sigma_comp} $\tilde\Sigma_n = \Phi^*\sigma^n\Phi$ for all $n \in \ZZ$.
\end{enumerate}
\end{prop}
\begin{proof}
Clearly
\[
\sigma \Phi f(\omega,n) = \Phi f(\omega,n+1) = f(\omega_{n+1}) = f(\pred(\omega_n)) = \Sigma f (\omega_n) = \Phi \Sigma f (\omega,n)
\]
for all $f \in \CC^T$, $\omega \in \Omega$ and $n \in \ZZ$, which proves part \ref{en:sigma_lift}. Iteration of this identity also gives
\[
\sigma^n \Phi = \Phi \Sigma^n
\]
for all $n \in \NN$. Applying $\Phi^*$ to both sides of this identity and using the fact that $\Phi^* \Phi = \id_T$ (see Proposition \ref{p: PhiPhi*} \ref{en:phiphibound}) gives
\[
\Phi^* \sigma^n \Phi = \Sigma^n,
\]
which proves part \ref{en:sigma_comp} in the case $n \in \NN$. To complete the proof of part \ref{en:sigma_comp}, it is enough to take adjoints in the latter identity, and use the fact that $(\sigma^n)^* = \sigma^{-n}$, as $\sigma^n$ is a unitary automorphism of $L^2(\nu \times \#)$.
\end{proof}

In light of the previous proposition, any operator $\cK$ on $\CC^T$ of the form
\begin{equation}\label{eq:liftable}
\cK = \sum_{n \in \ZZ} k(n) \, \tilde\Sigma_{-n},
\end{equation}
for some $k : \ZZ \to \CC$, ``lifts'' to an operator on $\CC^{\Omega \times \ZZ}$ of the form
\[
\sum_{n \in \ZZ} k(n) \, \sigma^{-n} = \id_\Omega \otimes \tau(k),
\]
where $\tau(k)$ is the convolution operator on $\ZZ$ with convolution kernel $k$, i.e., $\tau(k) f = f *_\ZZ k$.
In other words, we can write
\[
\cK = \Phi^* (\id_\Omega \otimes \tau(k)) \Phi.
\]
Therefore boundedness properties of $\cK$ can be related to boundedness properties of $\tau(k)$ by means of Proposition \ref{p: liftingbounds} and the following statement, which collects a few immediate consequences of Fubini's Theorem.

\begin{lemma}
Let $B$ be a linear operator on $\CC^\ZZ$.
\begin{enumerate}[label=(\roman*)]
\item For any $p \in [1,\infty]$, $B$ is $\ell^p(\ZZ)$-bounded if and only if $\id_\Omega \otimes B$ is $L^p(\nu \times \#)$-bounded, and their norms are the same.
\item For any $p \in [1,\infty]$, $B$ is of weak type $(1,1)$ on $\ZZ$ if and only if $\id_\Omega \otimes B$ is of weak type $(1,1)$ on $\Omega \times \ZZ$, and their norms are the same.
\end{enumerate}
\end{lemma}

Recall that $\Cv^p(\ZZ)$ is the space of all $L^p$-convolutors of $\ZZ$, i.e., the convolution kernels of the $\ell^p(\ZZ)$-bounded translation-invariant operators. By combining the previous results, we obtain the following statement.

\begin{theorem}\label{thm: transference}
For all $p \in [1,\infty]$, if $k \in \Cv^p(\ZZ)$ and $\cK$ is defined by \eqref{eq:liftable}, then $\cK$ is $L^p(\mu)$-bounded, with norm at most $\| k \|_{\Cv^p}$.
\end{theorem}

\begin{remark}\label{rem:transfernce}
The previous theorem can be thought of as a transference result for $L^p$ bounds from the group $\ZZ$ to the weighted tree $(T,\mu)$, which holds despite the fact that $n \mapsto \tilde\Sigma_n$ is not a representation of $\ZZ$ on $L^p(\mu)$, nor does it appear to fit into the more general framework of ``transference couples'' described in \cite{BPW}. It is not clear to us whether an analogous transference result could hold for weak type $(1,1)$ bounds: due to the obstruction discussed in Remark \ref{rem:lifting}, the proof given above for strong type bounds does not appear to extend to the weak type case too.
\end{remark}

\section{Boundedness results for \texorpdfstring{$\Rz$}{R}}\label{s: verticalRiesz}

\subsection{\texorpdfstring{$L^p$}{Lp}-boundedness of the Riesz transform \texorpdfstring{$\Rz$}{R}}

We start with an observation about ``almost-radial'' integral operators on $T$ in the sense of Section \ref{s: heat}.

\begin{lemma}\label{lem: radialintskew}
Let $\cK$ be an integral operator on $(T,\mu)$ with kernel
\[
K(x,y) = q^{-\ell(x)/2} \, G(d(x,y)) \, q^{-\ell(y)/2},
\]
where $G : \NN \to \RR$.
Let $\cS$ denote the composition $\nabla \cK$. Then,
\begin{equation}\label{skew part}
  \cS - \cS^* = \sum_{n \in \ZZ} h(n) \,\tilde{\Sigma}_{-n},
\end{equation}
where
\begin{equation}\label{eq: skewker}
h(n) = \begin{cases} \sgn(n) \, q^{(|n|-1)/2} \, [ G(|n|-1) - G(|n|+1) ] &\text{if } n \neq 0,\\
0 &\text{otherwise},
\end{cases}
\end{equation}
and $\tilde{\Sigma}_n$ is defined in \eqref{tildeSigma}.
\end{lemma}
\begin{proof}
Since $G$ is real-valued, $\cK$ is self-adjoint, so
\[
  \cS - \cS^* = \nabla \cK - \cK \nabla^* = -\Sigma \cK + \cK \Sigma^*.
\]
More explicitly, for every function $f$ on $T$,
\[\begin{split}
  \Sigma \cK f(x)
		&=\sum_{y \in T} q^{(\ell(y)-\ell(\pred(x)))/2} \, G(d(\pred(x),y)) \, f(y) \\
		&=\sum_{y \in T} q^{(\ell(y)-\ell(x))/2-1/2} \, G(d(\pred(x),y)) \, f(y),
\end{split}\]
and
\[\begin{split}
  \cK \Sigma^* f(x)
		&= \sum_{y \in T} q^{(\ell(y)-\ell(x))/2-1} \, G(d(x,y))\sum_{z \in \succ(y)} f(z) \\
		&= \sum_{z \in T} q^{(\ell(\pred(z))-\ell(x))/2-1} \, G(d(x,\pred(z))) \, f(z) \\
    &= \sum_{y \in T} q^{(\ell(y)-\ell(x))/2-1/2} \, G(d(x,\pred(y))) \, f(y),
\end{split}\]
thus
\[
 (\cS-\cS^*) f(x) =
  -\sum_{y \in T} q^{(\ell(y)-\ell(x))/2-1/2} \, [ G(d(\pred(x),y) - G(d(x,\pred(y))) ] \, f(y),
\]
and clearly $G(d(\pred(x),y) - G(d(x,\pred(y)))$ vanishes if $x \not < y$ or $y \not <x$. So we can restrict the sum to the set $\{y \in T \colon \ y < x \text{ or } x < y\}$.

Define now, for every $n \in \NN$, the set
\begin{equation}\label{eq:descendants}
\succ^n(x) = \{y \le x \colon d(x,y) = n\}
\end{equation}
of $n$th-generation descendants of $x$. Then
\[\begin{split}
  -(\cS - \cS^*) f(x)
		&= \sum_{n >0} q^{(n-1)/2} \, [ G(n-1) - G(n+1) ] \, f(\pred^n(x))\\
		&\quad +\sum_{n>0} q^{-(n+1)/2} \, [ G(n+1) - G(n-1) ] \sum_{y \in \succ^n(x)} f(y) \\
		&=\sum_{n>0} q^{(n-1)/2} \, [ G(n-1) - G(n+1) ] \, (\Sigma^n - (\Sigma^*)^n) f(x),
\end{split}\]
as required.
\end{proof}

We are now ready to prove our main result, Theorem \ref{riesztr}.

\begin{proof}[Proof of Theorem \ref{riesztr}]
By \cite[Theorem 2.3]{hs} and Proposition \ref{p: gradient_equivalence}, the Riesz transform $\Rz$ is bounded on $L^p(\mu)$ for $p \in (1,2]$. Recall now from Corollary \ref{c: negsquareroot} that $\opL^{-1/2}$ is an integral operator with kernel $K(x,y) = q^{-\ell(x)/2} \, G(d(x,y)) \, q^{-\ell(y)/2}$, with $G$ as in \eqref{eq: G}.
By applying Lemma \ref{lem: radialintskew} to $\cK = \opL^{-1/2}$ we deduce that
\begin{equation}\label{eq:skewsymm_lifting}
 \Rz - \Rz^* = \nabla \opL^{-1/2} - \opL^{-1/2} \nabla^* = \sum_{n \in \ZZ} \tilde{k}^\ZZ(n) \, \tilde{\Sigma}_{-n};
\end{equation}
for the last identity we used \eqref{eq: Gdiff} and \eqref{eq: skewker}, together with the fact that $\tilde k^\ZZ$ is odd.

Since, by Proposition \ref{p: rieszZZ}, $\tilde k^\ZZ$ is in $\Cv^p(\ZZ)$ for every $p\in (1,\infty)$, by Theorem \ref{thm: transference} we deduce that $\Rz-\Rz^*$ is bounded on $L^p(\mu)$ for every $p\in (1,\infty)$. By difference, we conclude that $\Rz^*$ is bounded on $L^p(\mu)$ for $p\in (1,2]$, or equivalently, that $\Rz$ is bounded on $L^p(\mu)$ for $p\in [2,\infty)$, as required.
\end{proof}

\begin{remark}\label{r: openproblem}
The identity \eqref{eq:skewsymm_lifting} shows that, in the notation of Section \ref{s: lifting},
\[
\Rz - \Rz^* = \Phi^* (\id_\Omega \otimes \widetilde \Rz_\ZZ) \Phi = \Phi^* (\id_\Omega \otimes (\Rz_\ZZ - \Rz_\ZZ^*)) \Phi;
\]
in other words, via the lifting procedure, the skew-symmetric part of $\Rz$ corresponds to the skew-symmetric part of $\Rz_\ZZ$.
As discussed in Remark \ref{rem:transfernce}, while we know that $\Rz_\ZZ$ and $\widetilde \Rz_\ZZ$ are of weak type $(1,1)$, via our transference strategy we appear not to be able to prove a weak type $(1,1)$ result for the operator $\Rz^*$, which remains an open problem.
\end{remark}

\subsection{Negative endpoint result for \texorpdfstring{$\Rz$}{R}}\label{ss: counterexample}
Hardy and BMO spaces adapted to the space $(T,\mu)$ were introduced and studied in \cite{ATV2, ATV1, LSTV}. These spaces are useful to obtain endpoint results for singular operators for $p=1$ and $p=\infty$, respectively, thanks to their good interpolation properties.

Let us recall that any subset $F$ of $T$ is called an \emph{admissible trapezoid} if it is either a singleton or can be written as
\begin{equation*}
  F = F_{h'}^{h''}(x_0) \defeq \lbrace x \in T \colon x \leq x_0, \, \ell(x_0) - h'' < \ell(x) \leq \ell(x_0) - h'\rbrace,
\end{equation*}
where $x_0$ is some vertex and $h', h''$ are two positive integers such that $2 \leq h''/h' \leq 12$. We denote by $\trF$ the family of admissible trapezoids.

A $(1,\infty)$-atom on $(T,\mu)$ is a mean-zero function supported on an admissible trapezoid $F$ and bounded by $\mu(F)^{-1}$. The atomic Hardy space $H^1(\mu)$ is the space of functions $g \in L^1(\mu)$ such that $g = \sum_{j} \lambda_j a_j$, where the $a_j$ are $(1,\infty)$-atoms and $\lbrace \lambda_j \rbrace$ is an $\ell^1$ sequence of complex numbers. The dual space of $H^1(\mu)$ can be identified with the space $BMO(\mu)$ \cite[Theorem 4.10]{LSTV}, which is defined as the space of functions $f$ on $T$ for which $\sup_{F \in \trF} |f-f_F|_F < \infty$, where $f_F$ denotes the integral mean of a function $f$ on the set $F$ with respect to the measure $\mu$. In particular, there exists a constant $C \in (0,\infty)$ such that, for any $(1,\infty)$-atom $a$,
\begin{equation}\label{eq: dual paring}
  |\langle f, a \rangle| = \left| \sum_{x \in T} f(x) \, a(x) \, \mu(x) \right| \leq C \Vert f \Vert_{BMO(\mu)} \qquad \forall f \in BMO(\mu).
\end{equation}

Admissible trapezoids are used as base sets for extending the Calder\'on--Zygmund theory developed in \cite{hs} to trees with locally doubling flow measures, playing the role balls play in the classical theory. In particular, the following lemma holds.

\begin{lemma}[{\cite[Theorem 5.8]{LSTV}}]\label{lem: hormander}
Let $\cK$ be a linear operator which is bounded on $L^2(\mu)$ and admits a
kernel $K$ satisfying the condition
\begin{equation}\label{eq: hormander}
  \sup_{F \in \trF} \sup_{y,z \in F} \sum_{x \notin F^*} |K(x,y)-K(x,z)| \, \mu(x) < + \infty,
\end{equation}
where, for any $F=F_{h'}^{h''}(x_0) \in \trF$, we define $F^{*}=\{x \in T : d(x,F)<h' \}$. Then $\cK$ extends to an operator which is of weak type (1,1), bounded from $H^1(\mu)$ to $L^1(\mu)$ and on $L^p(\mu)$, for $p\in (1,2)$.
If the kernel $K$ satisfies the condition
\begin{equation}\label{eq: hormanderdual}
  \sup_{F \in \trF} \sup_{y,z \in F} \sum_{x \notin F^*} |K(y,x)-K(z,x)| \, \mu(x) < + \infty,
\end{equation}
then $\cK$ extends to an operator which is bounded from $L^{\infty}(\mu)$ to $BMO(\mu)$ and on $L^p(\mu)$, for $p\in (2,\infty)$.
\end{lemma}

It is known that $\Rz$ is bounded from $H^1(\mu)$ to $L^1(\mu)$ \cite[Section 4.3]{ATV1}. We show below that $\Rz$ does not map $L^\infty(\mu)$ into $BMO(\mu)$. This can be thought of as a discrete counterpart to the counterexamples in the continuous setting discussed in \cite[Section 4]{SV2}.

\begin{prop}\label{h1l1}
The Riesz transform $\Rz$ does not map $L^\infty(\mu)$ into $BMO(\mu)$.
\end{prop}
\begin{proof}
By \eqref{eq: dual paring} it is enough to exhibit a function $f\in L^\infty(\mu)$ and a $(1,\infty)$-atom $a$ such that the dual pairing $\langle \Rz f, a\rangle$ is not bounded. Consider the admissible trapezoid $F = F_1^2(o) = \succ(o)$, with $\mu(F) = 1$. Pick $x_1, x_2 \in F$ such that $x_1 \ne x_2$ and define the $(1,\infty)$-atom $a=\delta_{x_1}-\delta_{x_2}$. Let $f = \chi_{\{x \colon x\leq x_1\}}$. Then,
\begin{equation*}
\begin{split}
  \langle \Rz f, a\rangle
		&= \Rz f(x_1) \, \mu(x_1) - \Rz f(x_2) \, \mu(x_2) \\
    &= q^{-1} \, [ \opL^{-1/2} f(x_1) - \opL^{-1/2} f(x_2) ],
\end{split}
\end{equation*}
where we used that $\Rz = \nabla \opL^{-1/2} = (\id_T - \Sigma) \opL^{-1/2}$, $\mu(x_1) = \mu(x_2) = 1/q$ and the cancellation induced from the fact that $\pred(x_1)=\pred(x_2)$. From Corollary \ref{c: negsquareroot} and the fact that $\ell(x_1) = \ell(x_2) = -1$, we then deduce that
\[
\langle \Rz f, a\rangle = q^{-1/2} \sum_{y \colon y \leq x_1} q^{\ell(y)/2} \,  [ G(d(x_1,y)) - G(d(x_2,y)) ].
\]
Next, observe that whenever $y\leq x_1$, we have $d(y,x_2) = d(y,x_1)+2$, and $-\ell(y) = d(y,x_1)+1$, so
\[\begin{split}
\langle \Rz f, a\rangle
&= q^{-1/2} \sum_{y \colon y \leq x_1} q^{\ell(y)/2} \, [ G(d(x_1,y)) - G(d(x_1,y)+2) ] \\
&= q^{-1} \sum_{n \geq 0} q^{n/2} \, [ G(n) - G(n+2) ] \\
&= q^{-1} \sum_{n \geq 0} \tilde k^\ZZ(n+1) = +\infty ,
\end{split}\]
by \eqref{eq: Gdiff} and \eqref{eq:rieszZZformulas}, and we are done.
\end{proof}

\begin{remark}\label{oss:noHorm}
By Proposition \ref{h1l1}, we deduce that the integral kernel of $\Rz$ does not satisfy the dual H\"ormander condition \eqref{eq: hormanderdual}. Indeed, otherwise, Lemma \ref{lem: hormander} would imply the $L^{\infty}(\mu) \to BMO(\mu)$ boundedness of $\Rz$. Notice that this phenomenon is in sharp contrast with the well known endpoint results for the Euclidean Riesz transforms of the first order, as well as the ones for the discrete first-order Riesz transforms on $\ZZ$ and more general finitely generated abelian groups \cite[Section 8]{HSC}, and it shows why it was not possible to use condition \eqref{eq: hormanderdual} to study the $L^p$-boundedness of $\Rz$ for $p \in (2,\infty)$.
\end{remark}

By Proposition \ref{h1l1} we deduce that $\Rz^*$ is not bounded from $H^1(\mu)$ to $L^1(\mu)$.
As it is an open question (see Remark \ref{r: openproblem}) whether $\Rz^*$ is of weak type $(1,1)$, no positive endpoint results for $p=1$ and $\Rz^*$ appear to be available. This partially motivates the introduction in the following section of another natural class of Riesz transforms associated with the flow Laplacian on $(T,\mu)$, for which we are able to prove the $L^p$-boundedness for $p \in (1,\infty)$, but also weak type $(1,1)$ endpoint results both for the operator and its adjoint.

\section{Horizontal Riesz transforms}\label{s: horizontal}

Let $\varepsilon \in \CC^T$ be bounded and such that $\Sigma^* \varepsilon = 0$ on $T$; in other words, we require that
\[
\sum_{y \in \succ(x)} \varepsilon(y) = 0 \qquad\forall x \in T.
\]
For every function $f$ in $\CC^T$ we define the $\varepsilon$-horizontal gradient $\hnabla f$ as
\[
  \hnabla f(x)=\Sigma^*(\varepsilon f)(x) = \frac{1}{q} \sum_{y \in \succ(x)} \varepsilon(y)f(y) \qquad \forall x\in T.
\]
We summarize some properties of the $\varepsilon$-horizontal gradient in the following proposition.

\begin{prop}\label{p: epsilongradient}
The following hold:
\begin{enumerate}[label=(\roman*)]
  \item\label{en:epsilonsigma} $\hnabla^* f = \overline{\varepsilon} \, \Sigma f$ for all $f \in \CC^T$;
  \item\label{en:epsilonnabla} $\hnabla = \hnabla \, \nabla$;
  \item\label{en:epsilonlp} for any $p \in [1,\infty]$,
\begin{align*}
\|\hnabla^*\|_{L^p(\mu)\to L^p(\mu)} &\leq \|\varepsilon\|_{\infty},\\
\|\hnabla^*\|_{L^{1,\infty}(\mu) \to L^{1,\infty}(\mu)} &\leq \|\varepsilon\|_{\infty};
\end{align*}
\item\label{en:epsilonslp} for any $p \in [1,\infty]$,
\begin{align*}
\|\hnabla\|_{L^{p}(\mu)\to L^{p}(\mu)} &\leq \|\varepsilon\|_{\infty},\\
\|\hnabla\|_{L^{1,\infty}(\mu)\to L^{1,\infty}(\mu)} &\leq q \|\varepsilon\|_{\infty};
\end{align*}
  \item\label{en:epsilonorth} $\imm(\hnabla^*) \perp \imm(\Sigma)$;
  \item\label{en:epsiloninner} for all $f,g\in \CC^T$ and $m,n \in \NN$,
\begin{equation}\label{ortrel}
  \langle \Sigma^n \hnabla^* f, \Sigma^m \hnabla^* g \rangle = \delta_{nm} \langle \hnabla^* f, \hnabla^* g \rangle.
\end{equation}
\end{enumerate}
\end{prop}
\begin{proof}
Part \ref{en:epsilonsigma} is immediately deduced from the definitions, as
\[
\hnabla^* = (\Sigma^* \varepsilon)^* = \overline{\varepsilon} \, \Sigma,
\]
where $\varepsilon$ and $\overline{\varepsilon}$ are thought of as multiplication operators.

As for part \ref{en:epsilonnabla}, for any function $f \in \CC^T$, since $\Sigma^* \varepsilon=0$,
\[
  \hnabla f(x) = \frac{1}{q} \sum_{y \in \succ(x)} \varepsilon(y) (f(y)-f(x)) = \hnabla \nabla f(x).
\]

Part \ref{en:epsilonlp} follows from Proposition \ref{SigmaSigma*} \ref{en:sigmasigmaiso} and the fact that $\hnabla^* = \overline\varepsilon \, \Sigma$. Similarly, part \ref{en:epsilonslp} follows from Proposition \ref{SigmaSigma*} \ref{en:sigmasigmabound} and the fact that $\hnabla = \Sigma^* \, \varepsilon$.

Now, for every function $f \in \CC^T$,
\[
  \hnabla \Sigma f = \Sigma^*( \varepsilon \, \Sigma f) = f \, \Sigma^* \varepsilon = 0,
\]
by Proposition \ref{SigmaSigma*} \ref{en:sigmasigmaproduct} and the assumption $\Sigma^* \varepsilon = 0$ on $\varepsilon$. This proves part \ref{en:epsilonorth}.

The orthogonality relation \ref{en:epsiloninner} is a consequence of \ref{en:epsilonorth} and the fact that $\Sigma$ is an isometric embedding on $L^2(\mu)$, by Proposition \ref{SigmaSigma*} \ref{en:sigmasigmaiso}.
\end{proof}

From the above proposition, we obtain an $L^2$-boundedness result for a class of operators on $(T,\mu)$.
The following result should be compared to the case $p=2$ of Theorem \ref{thm: transference}, where a similar class of operators is considered. Crucially, here we do not require that the sequence $F$ in the definition of the operator (see \eqref{def: P} below), once extended by zeros, is an $L^2$-convolutor on $\ZZ$, but only that it is square-summable. In other words, here we do not require any cancellations from $F$; the required cancellations yielding the $L^2$-boundedness of the resulting operator are instead provided by the orthogonality relations \eqref{ortrel}.

\begin{prop}\label{p: mathcal P}
Let $\cP$ be the linear operator on $L^2(\mu)$ defined by
\begin{equation}\label{def: P}
  \cP f = \sum_{n \geq 0} F(n) \,\Sigma^n \hnabla^* f
\end{equation}
for every $f \in L^2(\mu)$, where $F \in \ell^2(\NN)$. Then, $\cP$ is bounded on $L^2(\mu)$, with
\[
\|\cP\|_{L^2(\mu) \to L^2(\mu)} \leq \| F \|_{\ell^2(\NN)} \, \|\varepsilon\|_{\infty}.
\]
\end{prop}
\begin{proof}
Let $f$ be a function in $L^2(\mu)$. By \eqref{ortrel},
\[
  \|\cP f\|_{L^2(\mu)}^2 = \sum_{n \ge 0}|F(n)|^2 \, \| \hnabla^* f \|_{L^2(\mu)}^2 = \|F\|_{\ell^2(\NN)}^2 \, \|\hnabla^* f\|_{L^2(\mu)}^2,
\]
hence, by Proposition \ref{p: epsilongradient},
\[
 \|\cP f\|_{L^2(\mu)} \le \|F\|_{\ell^2(\NN)} \, \|\varepsilon\|_{\infty} \,  \|f\|_{L^2(\mu)},
\]
as desired.
\end{proof}

Interestingly enough, an adaptation of the above strategy also allows us to deduce the weak type $(1,1)$ boundedness of an operator of the form \eqref{def: P} from a non-cancellative assumption on $F$. The proof of the result below is significantly inspired by that of \cite[Theorem 3]{GS}.

\begin{theorem} \label{th:weak1P}
Let $\cP$ be as in \eqref{def: P} with $F \in \ell^{1,\infty}(\NN)$. Then, for all $f \in \CC^T$ and $\lambda > 0$,
\begin{equation}\label{wt1}
  \mu(\{|\cP f|>\lambda\}) \le 3 \, \|F\|_{\ell^{1,\infty}(\NN)} \, \|\varepsilon\|_{\infty} \frac{\|f\|_{L^1(\mu)}}{\lambda}.
\end{equation}
\end{theorem}
\begin{proof}
Note that $\hnabla^*$ and $\cP$ depend $\RR$-linearly on $\varepsilon$; hence, without loss of generality, we may assume that $\|\varepsilon\|_\infty = 1$.

Let $\lambda>0$ and $f \in L^{1}(\mu)$. For any $n \in \NN$, decompose $f=f_n+\tilde{f}_n$ where $f_n = f \chi_{\{|F(n)f|>\lambda\}}$. Then,
\begin{multline}\label{eq:firstsplit}
\mu(\{|\cP f|>\lambda\}) \\
\le \mu\biggl(\biggl\{\biggl|\sum_{n \ge 0} F(n) \Sigma^n \hnabla^* f_n \biggr| > 0\biggr\} \biggr)
+\mu\biggl(\biggl\{\biggl| \sum_{n \geq 0} F(n) \Sigma^n \hnabla^* \tilde{f}_n \biggr| > \lambda\biggr\} \biggr).
\end{multline}

Now, $\{f_n \ne 0\} = \{|F(n)f| > \lambda\}$ and
\begin{equation}\label{eq:set_containments}
\begin{split}
  \{ \Sigma^n \hnabla^*f_n \ne 0\}
		&= \pred^{-n} \{ \hnabla^* f_n \ne 0\}
		= \pred^{-n} \{ \overline\varepsilon \, \Sigma f_n \ne 0\} \\
		&\subseteq \pred^{-n-1} \{f_n \ne 0\}
		= \pred^{-n-1} \{|F(n)f| > \lambda\},
\end{split}
\end{equation}
whence
\begin{equation}\label{stimamenoprecisa}
\begin{split}
    \mu\biggl( \biggl\{ \biggl|\sum_{n \ge 0} F(n) \Sigma^n \hnabla^* f_n \biggr| > 0 \biggr\} \biggr)
		&\le \sum_{n \ge 0} \mu(\{ \Sigma^n \hnabla^* f_n \ne 0 \}) \\
		&\le \sum_{n \ge 0} \mu(\{ |F(n)f| > \lambda \});
\end{split}
\end{equation}
in the last inequality we used \eqref{eq:set_containments} and the fact that, since $\mu$ is a flow measure,
\[
\mu(\pred^{-k}(E)) = \mu(E)
\]
for any $E \subset T$ and $k \in \NN$. On the other hand, by Fubini's Theorem,
\begin{equation}\label{eq:firstpart2}
\begin{split}
    \sum_{n \geq 0} \mu(\{ |F(n)f| > \lambda \})
		&= \sum_{x \in T} \mu(x) \,\# \{n \in \NN \colon |F(n)f(x)| > \lambda\} \\
		&\le \frac{\|F\|_{\ell^{1,\infty}(\NN)}}{\lambda} \sum_{x \in T} \mu(x) \, |f(x)|\\
		&= \frac{\|F\|_{\ell^{1,\infty}(\NN)}}{\lambda} \|f\|_{L^1(\mu)}.
\end{split}
\end{equation}

For the remaining part, Chebyshev's inequality and \eqref{ortrel} imply that
\begin{multline}\label{eq:secondpart1}
  \mu\biggl(\biggl\{ \biggl| \sum_{n \ge 0} F(n) \Sigma^n \hnabla^* \tilde{f}_n \biggr| > \lambda \biggr\} \biggr)
		\le \frac{1}{\lambda^2} \biggl\| \sum_{n \geq 0} F(n) \Sigma^n \hnabla^* \tilde{f}_n \biggr\|_{L^2(\mu)}^2 \\
		= \frac{1}{\lambda^2} \sum_{n,m \geq 0} F(n) \overline{F(m)} \, \langle \Sigma^n \hnabla^* \tilde{f}_n, \Sigma^m \hnabla^* \tilde{f}_m \rangle
		= \frac{1}{\lambda^2} \sum_{n \geq 0} |F(n)|^2 \|\hnabla^* \tilde{f}_n\|_{L^2(\mu)}^2.
\end{multline}
We now observe that, for all $n \in \NN$,
\[
  \hnabla^* \tilde{f}_n = \overline{\varepsilon} \, \Sigma \tilde{f}_n = \overline{\varepsilon} \, (\Sigma f) \, \chi_{\{|F(n)\Sigma f| \le \lambda\}};
\]
hence $| \hnabla^* \tilde{f}_n | \le |\Sigma f| \, \chi_{\{|F(n) \Sigma f| \le \lambda\}}$ (recall that $\|\varepsilon\|_{\infty} = 1$), and therefore
\[
  \|\hnabla^* \tilde{f}_n\|_{L^2(\mu)}^2
		= \sum_{x \in T} \mu(x) \, |\hnabla^* \tilde{f}_n(x)|^2
		\le \sum_{x \colon \Sigma f(x) \ne 0} \mu(x) \, |\Sigma f(x)|^2 \chi_{\{|F(n)\Sigma|\le \lambda\}}(x).
\]
Thus
\begin{equation}\label{eq:secondpart2}
\begin{split}
  \sum_{n \geq 0} |F(n)|^2 \|\hnabla^* \tilde{f}_n\|_{L^2(\mu)}^2
		&\le \sum_{x \colon \Sigma f(x) \ne 0} \mu(x) \, |\Sigma f(x)|^2 \sum_{n \colon |F(n)\Sigma f(x)|\le \lambda} |F(n)|^2 \\
    &\le 2 \lambda \|F\|_{\ell^{1,\infty}(\NN)}  \sum_{x \in T} \mu(x) \, |\Sigma f(x)| \\
		&= 2 \lambda \|F\|_{\ell^{1,\infty}(\NN)} \, \|\Sigma f\|_{L^1(\mu)};
\end{split}
\end{equation}
in the last inequality we used the fact that, for all $\lambda > 0$,
\[\begin{split}
    \sum_{n\ge 0} |F(n)|^2 \chi_{\{|F(n)|\le \lambda\}}
		&=\int_0^\infty \# \{n \in \NN \colon |F(n)|^2 \chi_{\{|F(n)|\le \lambda\}}>\alpha\} \,d\alpha \\
		&\le \int_0^{\lambda^2} \# \{n \in \NN \colon |F(n)|>\alpha^{1/2}\} \,d\alpha \\
		&\le\|F\|_{\ell^{1,\infty}(\NN)}\int_0^{\lambda^2}\alpha^{-1/2} \,d\alpha\\
		&=2\lambda \|F\|_{\ell^{1,\infty}(\NN)}.
\end{split}\]

The desired estimate follows by combining \eqref{eq:firstsplit}, \eqref{stimamenoprecisa}, \eqref{eq:firstpart2}, \eqref{eq:secondpart1} and \eqref{eq:secondpart2}.
\end{proof}

The relevance of the above bounds is made clear by the following computation, which should be compared to Lemma \ref{lem: radialintskew}.

\begin{lemma}\label{lem:halfkernel}
Let $\cK$ be an integral operator on $(T,\mu)$ whose integral kernel has the form
\[
K(x,y) = q^{-\ell(x)/2} \, G(d(x,y)) \, q^{-\ell(y)/2}
\]
for some $G : \NN \to \CC$. Then,
\[
\cK \hnabla^* = \sum_{n \geq 0} q^{n/2} \,[ G(n) - G(n+2) ] \, \Sigma^n \, \hnabla^*.
\]
\end{lemma}
\begin{proof}
For all $f \in \CC^T$ and $x \in T$,
\begin{equation}\label{canc}
\begin{split}
 \cK \hnabla^* f(x)
	&= \sum_{y \colon \pred(y)>x} q^{-\ell(x)/2} \, G(d(x,y)) \, q^{\ell(y)/2} \, \overline\varepsilon(y) \, f(\pred(y)) \\
	&\quad +\sum_{y \colon \pred(y)\not>x} q^{-\ell(x)/2} \, G(d(x,y)) \, q^{\ell(y)/2} \, \overline\varepsilon(y)f(\pred(y)).
\end{split}
\end{equation}
The second sum in \eqref{canc} is equal to zero: indeed, if $\pred(y) \not> x$, then $d(x,y) = d(x,\pred(y))+1$, thus
\begin{multline*}
\sum_{y \colon \pred(y)\not>x} q^{-\ell(x)/2} \, G(d(x,y)) \, q^{\ell(y)/2} \, \overline\varepsilon(y)f(\pred(y)) \\
= \sum_{z \colon z\not>x} q^{-\ell(x)/2} \, G(d(x,z)+1) \, q^{(\ell(z)-1)/2} f(z) \sum_{y \in \succ(z)} \overline\varepsilon(y)
\end{multline*}
and $\Sigma^*\overline\varepsilon=0$. It follows that
\begin{equation}\label{eq:firstpartkernelsemi}
 \cK \hnabla^* f(x)
  = \sum_{z \colon z \ge x} q^{-\ell(x)/2} \left[\sum_{y \colon \pred(y)=\pred(z)} G(d(x,y)) \, \overline\varepsilon(y) \right] \, q^{\ell(z)/2} \, f(\pred(z)) .
\end{equation}
We now observe that, for all $z,y \in T$, if $z \geq x$ and $z \neq y \in \succ(\pred(z))$, then $d(x,y) = d(x,z)+2$, and moreover
\[
 \sum_{y \in \succ(\pred(z)), y \ne z}\overline\varepsilon(y) = -\overline\varepsilon(z),
\]
because $\Sigma^* \varepsilon = 0$; as a consequence,
\[
\sum_{y \colon \pred(y)=\pred(z)} G(d(x,y)) \, \overline\varepsilon(y) = \overline\varepsilon(z) \, [G(d(x,z))-G(d(x,z)+2)].
\]
From \eqref{eq:firstpartkernelsemi} we then deduce that
\[\begin{split}
 \cK \hnabla^* f(x)
  &= \sum_{z \colon z \ge x} q^{-\ell(x)/2} \, [G(d(x,z))-G(d(x,z)+2)] \, q^{\ell(z)/2} \, \overline\varepsilon(z) \, f(\pred(z)) \\
  &= \sum_{n \geq 0} q^{n/2} \, [G(n)-G(n+2)] \, \hnabla^* f(\pred^n(x)),
\end{split}\]
as desired.
\end{proof}

As discussed in the introduction, we define the $\varepsilon$-horizontal Riesz transform by
\[
\Rz_\varepsilon = \hnabla \, \opL^{-1/2}.
\]
From Proposition \ref{p: epsilongradient} \ref{en:epsilonnabla} we deduce that
\[
\Rz_\varepsilon = \hnabla \, \Rz.
\]
Since $\hnabla$ is bounded on $L^{1,\infty}(\mu)$ and on $L^p(\mu)$ for every $p \in [1,\infty]$ (see Proposition \ref{p: epsilongradient} \ref{en:epsilonlp}), any weak type $(1,1)$ and $L^p$-boundedness property for $\Rz$ transfers to $\Rz_\varepsilon$. In particular, from Theorem \ref{riesztr} we deduce that $\Rz_\varepsilon$ is bounded on $L^p(\mu)$ for every $p\in (1,\infty)$.
Moreover, since $\Rz$ is of weak type $(1,1)$, $\Rz_\varepsilon$ is also of weak type $(1,1)$.

An analogous argument applies to the adjoint operators $\Rz^*$ and $\Rz_\varepsilon^* = \Rz^* \hnabla^*$, as $\hnabla^*$ is $L^p(\mu)$-bounded for all $p \in [1,\infty]$ (see Proposition \ref{p: epsilongradient} \ref{en:epsilonslp}). Recall that we do not know (see Remark \ref{r: openproblem}) whether $\Rz^*$ is of weak type $(1,1)$.
Nevertheless, we are able to prove a weaker result, namely, the weak type $(1,1)$ boundedness of $\Rz_{\varepsilon}^*$, which can be considered as a discrete counterpart of \cite[Theorem 1]{GS}.

\begin{theorem}\label{t: Respilonstar}
The operator $\Rz_\varepsilon^*$ is of weak type $(1,1)$.
\end{theorem}
\begin{proof}
In light of Corollary \ref{c: negsquareroot}, we can apply Lemma \ref{lem:halfkernel} with $\cK = \opL^{-1/2}$ and $G$ given by \eqref{eq: G}; thus, by \eqref{eq: Gdiff},
\begin{equation}\label{Repsilonstar}
\Rz_{\varepsilon}^* = \opL^{-1/2} \hnabla^* = \sum_{n\ge 0} \tilde{k}^{\ZZ}(n+1) \, \Sigma^n \hnabla^* .
\end{equation}
On the other hand, by \eqref{CZktilde}, $\tilde{k}^{\ZZ}(1+\cdot)|_\NN$ belongs to $\ell^{1,\infty}(\NN)$, so the desired bound follows by Theorem \ref{th:weak1P}.
\end{proof}

We point out that, as was the case for the Riesz transform $\Rz$ (see Remark \ref{oss:noHorm}), the previous weak type endpoint result cannot be deduced by showing that the integral kernel of $\Rz_\varepsilon$ satisfies the dual H\"ormander condition \eqref{eq: hormanderdual}. This is a consequence of the following negative endpoint result for the horizontal Riesz transforms, analogous to the one for $\Rz$ discussed in Section \ref{ss: counterexample}.

\begin{prop}\label{p:horiz_h1l1}
If $\varepsilon$ is not identically zero, then $\Rz_\varepsilon^*$ does not map $H^1(\mu)$ into $L^1(\mu)$.
\end{prop}
\begin{proof}
As $\varepsilon \not\equiv 0$, there exists $x_1 \in T$ such that $\varepsilon|_{\succ(x_1)} \not\equiv 0$.
Let $\bar x = \pred(x_1)$ and take $x_2 \in \succ(\bar x) \setminus \{ x_1 \}$.

Much as in the proof of Proposition \ref{h1l1}, consider the admissible trapezoid $F = F_1^2(\bar x) = \succ(\bar x)$, 
and define the $(1,\infty)$-atom $a = \mu(F)^{-1} (\delta_{x_1} - \delta_{x_2})$ supported in $F$. Then,
$\hnabla^*a = \overline{\varepsilon} \, \Sigma a$ by Proposition \ref{p: epsilongradient}. In particular, for any $y \in \succ(x_1)$, $\hnabla^*a(y) = \mu(F)^{-1} \, \overline{\varepsilon}(y)$, and therefore $\hnabla^* a \not\equiv 0$, because $\varepsilon|_{\succ(x_1)} \not\equiv 0$.

From the identity $\hnabla^*a = \overline{\varepsilon} \, \Sigma a$ we also deduce that $\supp(\hnabla^* a) \subseteq \succ^2(\bar x)$, and therefore $\supp(\Sigma^n \hnabla^* a) \subseteq \succ^{n+2}(\bar x)$; here we are using the notation $\succ^n(x)$ from \eqref{eq:descendants}. In particular, the supports of the functions $\Sigma^n \hnabla^* a$, $n \in \NN$, are pairwise disjoint. As $\Sigma$ preserves $L^1(\mu)$-norms (see Proposition \ref{SigmaSigma*}), from \eqref{Repsilonstar} we conclude that
\[
\|\Rz_\varepsilon^* a\|_{L^1(\mu)} = \sum_{n\ge 0} \tilde{k}^{\ZZ}(n+1) \|\Sigma^n \hnabla^* a\|_{L^1(\mu)} = \|\hnabla^* a\|_{L^1(\mu)} \sum_{n\ge 0} \tilde{k}^{\ZZ}(n+1) = +\infty,
\]
where we used Proposition \ref{p: rieszZZ} and the fact that $\hnabla^* a \not\equiv 0$.
\end{proof}

\providecommand{\bysame}{\leavevmode\hbox to3em{\hrulefill}\thinspace}

\end{document}